\documentclass[reqno,A4paper]{amsart}
\usepackage{amsmath}
\usepackage{amssymb}
\usepackage{amsthm}
\usepackage{enumerate}
\usepackage{mathrsfs}
\usepackage{eqlist}
\usepackage{array}
\usepackage{amsmath,amssymb,url}
\usepackage{tikz}
\usepackage{enumitem}

\numberwithin{equation}{section}

\theoremstyle{plain}
\newtheorem{theorem}{Theorem}[section]
\newtheorem{lemma}[theorem]{Lemma}
\newtheorem{proposition}[theorem]{Proposition}
\newtheorem{corollary}[theorem]{Corollary}

\theoremstyle{definition}
\newtheorem{definition}[theorem]{Definition}
\newtheorem{notation}[theorem]{Notation}
\newtheorem{remark}[theorem]{Remark}
\newtheorem{example}[theorem]{Example}

\begin{document}
	
	
	\title[Simple and sub-directly irreducible dBas]{Simple and sub-directly irreducible double Boolean algebras}
	
	
	
	\author[G. T. Kembang \and L. Kwuida \and E. R. A. Temgoua \and Y. L. J. Tenkeu]{G.T. Kembang* \and L. Kwuida** \and E.R.A. Temgoua*** \and Y.L.J. Tenkeu****}
	\newcommand{\acr}{\newline\indent}
	
	\address{\llap{*\,}Department of Mathematics\acr
		Faculty of Sciences\acr
		University of Yaounde 1\acr
		Yaounde\acr
		Cameroon}
	\email{tenkeugael@gmail.com}
	\address{\llap{**\,}Department of Business\acr
		Bern University of Applied Science\acr
		Bern\acr
		Switzerland}
	\email{leonard.kwuida@bfh.ch}
	\address{\llap{***\,}Department of Mathematics\acr
		Ecole Normale Supérieure de Yaoundé\acr
		University of Yaounde 1\acr
		Yaounde\acr
		Cameroon}
	\email{retemgoua@gmail.com}
	\address{\llap{****\,}Department of Mathematics\acr
		Faculty of Sciences\acr
		University of Yaounde 1\acr
		Yaounde\acr
		Cameroon}
	\email{ytenkeu2018@gmail.com}
	
	\thanks{This work was supported by Swiss National Science Foundation(SNF)
		Grant No Nr.: IZSEZO-219516/1}
	\subjclass{06B15, 06D50, 06E15, 06E75}
	\keywords{ Double Boolean algebra, Congruence, Sub-directly irreducible algebra, Simple algebra, Glued sum}

	
	


	\begin{abstract}
		Double Boolean algebras   are algebras $\underline{D}=(D;\sqcap,\sqcup,\neg,\lrcorner,\bot,\top)$ of type $(2,2,1,1,0,0)$
		introduced by Rudolf Wille to capture the equational theory of the algebra of protoconcepts. Every double Boolean algebra $\underline{D}$ contains two Boolean algebras denoted by $\underline{D}_{\sqcap}$ and $\underline{D}_{\sqcup}$. A double Boolean algebra $\underline{D}$ is said  pure if $D=D_{\sqcap}\cup D_{\sqcup}$, and trivial if $\bot\sqcup\bot=\top\sqcap\top$. In this work, we first show that a double Boolean algebra is pure and trivial if and only if it is a glued sum of two Boolean algebras; secondly, we characterize simple double Boolean algebras; and finally, we determine  up to isomorphism all  sub-directly irreducible algebras of some sub-classes of the variety of double Boolean algebras. 
	\end{abstract}
	
	\maketitle
	
	
	\section{Introduction}\label{sec:intro}
	
	Formal concept analysis constitutes a mathematical framework for
	knowledge representation and reasoning \cite{13}. In order to develop Boolean concept logic, which is a logic based on concepts as units of thought, it was necessary to define the negation of a concept. The first approach requires the negation of a concept to be a concept, and gives rise to two operations, a weak negation $^\triangle$ and a weak opposition $^\bigtriangledown$, 
	which lead to weakly dicomplemented lattices~\cite{6}. The second approach generalizes the notion of concept to that of protoconcept, with a goal to keep a correspondence between "negation" and "set complementation". It gives rise to a negation $"\neg"$ and an opposition$"\lrcorner"$. Protoconcepts lead to a new class of algebras called double Boolean algebra. Rudolf Wille showed that every double Boolean algebra can be quasi-embedded into a protoconcept algebra~\cite{14}. This representation theorem shows that the axioms of double Boolean algebras determine the equational theory of protoconcept algebras.  
	Since then, several researchers have focused on the study of this new algebraic structure: Vormbrock \cite{11},  Kwuida \cite{7}, Balbiani \cite{1}, Tenkeu  et al \cite{10}, Prosenjit  and Mohua  \cite{8, 9}. Given a variety $V$ of universal algebras, the determination of its simple and sub-directly irreducible elements plays a central role in its study because according to one of Birkhoff's theorem, any element of $V$ is isomorphic to a sub-direct product of sub-directly irreducible elements of $V$. 
	In \cite{11},  Vormbrock obtains that a finite double Boolean algebra $\underline{D}$ is sub-directly irreducible if and only if $\underline{D}$ is simple. But the characterization of simple double Boolean algebras is not yet known. To continue this investigation, we firstly characterize pure and trivial double Boolean algebras as glued sum of two Boolean algebras. Then, we characterize simple double Boolean algebras, and determine all sub-directly irreducible double Boolean algebras of some special sub-varieties of the variety of double Boolean algebras.
	
	Our work is organized as follows: Section \ref{sec2} recalls some basic notions from universal algebra and double Boolean algebras, necessary to understand our contribution. In  Section 3, we characterize  pure and trivial double Boolean algebras as well as 
	simple double Boolean algebras. In Section 4, we give the complete list of \textit{two}-element double Boolean algebras, and  determine, up to isomorphism, all  sub-directly irreducible double Boolean algebras of some sub-varieties of this variety. 
	The last section concludes the paper.
	
	\section{Preliminaries}\label{sec2}
	
	Let us recall some definitions and properties useful for the comprehension of this work. We begin with some tools from universal algebra, taken from \cite{3}. 
	
	\subsection{Some basic notions on universal algebras} 
	
	\begin{definition}
		Let $\underline{A}$ be an algebra of type $\mathcal{F}$ and let $\theta$ be an equivalence relation on $A$.
		$\theta$ is called a \textit{congruence relation} on $\underline{A}$ if $\theta$ satisfies the following compatibility property:
		for each $n$-ary function symbol $f\in\mathcal{F}$ and elements $a_i$, $b_i\in A$, if $a_i\theta b_i$ holds for $1 \le i \le n$, then $f^{A}(a_1,\cdots,a_n)\theta f^{A}(b_1,\cdots,b_n)$.
	\end{definition}
	
	\begin{notation}
		For an algebra $\underline{A}$, we denote by $Con(\underline{A})$ the set of all congruence relations on $\underline{A}$. 
	\end{notation}
	
	\begin{proposition}\cite[Theorem 5.3, p. 40]{3}
		If $\underline{A}$ is an algebra, then $(Con(\underline{A}),\subseteq)$ forms a complete lattice with $\Delta_{A}$ and $\nabla_{A}$ the smallest and the largest congruence relation respectively, where $\Delta_{A}=\{\, (x,x) \mid x\in A \,\}$ and $\nabla_{A}= A^{2}$. 
	\end{proposition}
	
	\begin{definition}
		An algebra $\underline{A}$ is said  \textit{simple} if $Con(\underline{A})=\{\Delta_{A}, \nabla_{A} \}.$
	\end{definition}
	
	\begin{definition}
		An algebra $\underline{A}$ is said \emph{congruence-distributive} if $Con(\underline{A})$ is a distributive lattice. 	
		A class $K$ of algebras is said \emph{congruence-distributive} if every member of $K$ is congruence-distributive.  
	\end{definition}
	
	
	
	\begin{definition}
		
		\begin{enumerate}
			\item  An algebra $\underline{A}$ is called a \emph{sub-direct product} of an indexed family $(\underline{A}_{i})_{i\in I}$ of algebras if:
			
			\begin{enumerate}[label=\textup{(\alph*)}]
				\item  $\underline{A}$ is a sub-algebra of $\prod\limits_{i\in I}\underline{A}_{i};$
				\item  $\pi_{i}(\underline{A})=\underline{A}_{i}$ for each $i\in I.$
			\end{enumerate}
			\item  An embedding $\alpha:\underline{A}\longrightarrow\prod\limits_{i\in I}\underline{A}_{i}$ is said \emph{sub-direct} if $\alpha(A)$ is a sub-direct product of the $(\underline{A}_{i})_{i\in I}$.
			\item  An algebra $\underline{A}$ is said \emph{sub-directly irreducible} if for every sub-direct embedding $\alpha : \underline{A}\longrightarrow \prod\limits_{i\in I}\underline{A}_{i}$, there is an $i\in I$ such that $\pi_{i}~\circ{}~\alpha : \underline{A}\longrightarrow \underline{A}_{i}$ is an isomorphism.
		\end{enumerate}
		
	\end{definition}
	
	\begin{proposition}\cite[Theorem 8.4, p. 63]{3}\label{t4} 
		An algebra $\underline{A}$ is sub-directly irreducible iff $\underline{A}$ is the \textit{one}-element algebra or there is a minimum
		congruence in $Con(\underline{A})\backslash\{\Delta_{A} \}$. In the latter case the minimum element is $\bigcap(Con (\underline{A})\backslash\{\Delta_{A} \}).$
	\end{proposition}
	
	\begin{notation}
		We denote by $\mathit{\underline{2}}$ the $two$-element Boolean algebra.
	\end{notation}
	\begin{lemma}\cite[Corollary 1.12, p. 134]{3}\label{l3} 
		The Boolean algebra $\textit{\underline{2}}$	is, up to isomorphism, the only sub-directly irreducible Boolean
		algebra having more than one element.
	\end{lemma}
	
	\begin{definition}\label{d5}
		Let $\underline{B}=(B; \wedge, \vee, ~'~, 0, 1)$ be a Boolean algebra.
		A subset $I$ of $B$ is called an \emph{ideal}  if it satisfies the following three conditions :
		\begin{enumerate}[label=\textup{(\alph*)}]
			\item  $0\in I$,
			\item  $\forall x,y\in B$, $x,y\in I\implies x\vee y\in I$,
			\item  $\forall x,y\in B$, $y\in I$, $x\le y\implies x\in I$.
		\end{enumerate}
		
		The notion of filter is defined dually.
	\end{definition}
	
	\begin{notation}
		Let $\underline{B}=(B; \wedge, \vee, ~'~, 0, 1)$ be a Boolean algebra. We denote by $\mathcal{I}(\underline{B})$ (resp. $\mathcal{F}(\underline{B})$ ) the set of all ideals (resp.  filters) of $\underline{B}$.
	\end{notation}
	
	\begin{proposition}\cite[Lemma 3.10, p. 148]{3} \label{p10} Let $\underline{B}$ be a Boolean algebra. Then $(Con(\underline{B}),\subseteq)$, $(\mathcal{I}(\underline{B});\subseteq)$ and $(\mathcal{F}(\underline{B});\subseteq)$ are isomorphic distributive lattices. 
	\end{proposition}
	
	Now, we end this section with some preliminaries on double Boolean algebras. For further information, we refer the reader to \cite{7, 8, 13, 14}.
	\subsection{Double Boolean algebras}
	
	\begin{definition}
		An algebra $\underline{D}=(D;\sqcap,\sqcup,\neg,\lrcorner,\bot,\top)$ of type $(2,2,1,1,0,0)$ is called \emph{double Boolean algebra (dBa for short)} if it satisfies the following  identities : 
		\[ \begin{array}{clcl}
			(1a)& (x\sqcap x)\sqcap y= x\sqcap y&(1b)& (x\sqcup x)\sqcup y= x\sqcup y\\
			(2a)& x\sqcap y=y\sqcap x& (2b)& x\sqcup y=y\sqcup x\\
			(3a)& x\sqcap(y\sqcap z)=(x\sqcap y)\sqcap z& (3b)& x\sqcup(y\sqcup z)=(x\sqcup y)\sqcup z\\
			(4a)& x\sqcap(x\sqcup y)=x\sqcap x& (4b)& x\sqcup(x\sqcap y)=x\sqcup x\\
			(5a)& x\sqcap(x\vee y)=x\sqcap x& (5b)& x\sqcup(x\wedge y)=x\sqcup x\\
			(6a)& x\sqcap (y\vee z)= (x\sqcap y)\vee (x\sqcap z)& (6b)& x\sqcup (y\wedge z)= (x\sqcup y)\wedge (x\sqcup z)\\
			(7a)& \neg\neg (x\sqcap y)= x\sqcap y& (7b)& \lrcorner\lrcorner (x\sqcup y)=x\sqcup y\\
			(8a)& \neg (x\sqcap x)= \neg x& (8b)& \lrcorner (x\sqcup x)=\lrcorner x\\
			(9a)& x\sqcap\neg x=\bot & (9b)& x\sqcup\lrcorner x=\top\\
			(10a)& \neg\bot=\top\sqcap\top& (10b)& \lrcorner\top=\bot\sqcup\bot\\
			(11a)&\neg\top=\bot& (11b)& \lrcorner\bot=\top\\
		\end{array}\]
		$$(12)~ (x\sqcap x)\sqcup (x\sqcap x)=(x\sqcup x)\sqcap (x\sqcup x)$$ 
		where $x\vee y=\neg(\neg x\sqcap\neg y)$ and $x\wedge y=\lrcorner(\lrcorner x\sqcup\lrcorner y).$
	\end{definition}
	
	On a dBa $\underline{D}=(D;\sqcap, \sqcup, \neg, \lrcorner, \bot, \top)$, a  relation "$\sqsubseteq$" is defined as follows:
	\[x\sqsubseteq y :\iff x\sqcap y=x\sqcap x~\text{and}~x\sqcup y=y\sqcup y.\]
	
	The relation "$\sqsubseteq$" is a \emph{quasi-order} on $D$.
	
	\begin{definition}\label{d1}
		A dBa $\underline{D}=(D;\sqcap,\sqcup,\neg,\lrcorner,\bot,\top)$ is called :
		\begin{enumerate}
			\item   \emph{pure} if for all $x\in D$, either $x\sqcap x=x$ or $x\sqcup x=x$;
			\item   \emph{trivial}\footnote{In universal algebra, a trivial algebra is a \textit{one}-element algebra. But in this work a trivial dBa will always refer to $(2)$ of Definition \ref{d1}. } if  $\top\sqcap \top=\bot\sqcup \bot$;
			\item   \emph{regular (or contextual)} if the quasi-order "$\sqsubseteq$" is an order relation.
		\end{enumerate}
	\end{definition}
	
	\begin{notation}
		We set:	$D_{\sqcap}:=\{x\in D \mid x\sqcap x= x \}$, $D_{\sqcup}:=\{x\in D \mid x\sqcup x= x \}$ and $D_{p}=D_{\sqcap}\cup D_{\sqcup}.$
	\end{notation}
	
	\begin{proposition}\cite[Theorem 1, p. 6]{11}\label{p8} Let $\underline{D}$ be a dBa.
		\begin{enumerate}
			\item  $\underline{D}_{\sqcap}:=(D_{\sqcap}; \sqcap, \vee, \neg, \bot, \neg\bot)$ is a Boolean algebra whose order relation is the restriction of $\sqsubseteq$ to $\underline{D}_{\sqcap}$  and is denoted by $\sqsubseteq_{\sqcap}$.
			\item  $\underline{D}_{\sqcup}:=(D_{\sqcup};\wedge, \sqcup, \lrcorner,\lrcorner\top, \top)$ is a Boolean algebra whose order relation is the restriction of $\sqsubseteq$ to $\underline{D}_{\sqcup}$  and is denoted by $\sqsubseteq_{\sqcup}$.
			\item  For any $x,y\in D$, $x\sqsubseteq y$ if and only if $x\sqcap x\sqsubseteq y\sqcap y$ and $x\sqcup x\sqsubseteq y\sqcup y.$
			\item  $\underline{D}_{p}=(D_{p}; \sqcap, \sqcup, \neg, \lrcorner, \bot, \top)$ is the largest pure sub-algebra of $\underline{D}$.
		\end{enumerate}
	\end{proposition}
	
	
	\begin{notation}
		For a dBa $\underline{D}$, $X\subseteq D_{\sqcap}$ (resp. $Y\subseteq D_{\sqcup}$), we denote by $I(X)$ (resp. $F(Y)$) the ideal of $\underline{D}_{\sqcap}$ (resp. filter of $\underline{D}_{\sqcup}$) generated by $X$ (resp. $Y$). If $X=\{x\}$ (resp. $Y=\{ y\}$), then $I(X)$ (resp. $F(Y)$) is denoted by $I(x)$ (resp. $F(y)$).
	\end{notation}
	\begin{example}\label{e1}\text{}
		
		\begin{enumerate}
			\item  The algebra  $\underline{D}_{3,I}:=(D_{3,I};\sqcap,\sqcup,\neg,\lrcorner,\bot,\top)$ with $D_{3,I}:=\{\bot, a, \top \}$ is defined by the Cayley tables in  Figure~\ref{OD3I}. It is a \emph{pure, trivial and regular dBa}. Its Hasse diagram is given by Figure~\ref{D3I}. Moreover,
			$D_{3,I_\sqcap}=\{\bot, a\}$ and $D_{3,I_\sqcup}=\{a,\top\}$. 
			
			\begin{figure}[ht]
				\begin{center}
					\begin{tikzpicture}[scale=.5]
						\node (one) at (0,1.5) {$\top$};
						\node (a) at (0,0) {$a$};
						\node (zero) at (0,-1.5) {$\bot$};
						\draw (zero)--(a)--(one);
					\end{tikzpicture}
					\caption{ A \emph{three}-element pure, trivial and regular dBa}\label{D3I}
				\end{center}
			\end{figure}
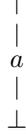 
			
			\begin{figure}[ht]
				\begin{tabular}{|c|ccc|}
					\hline 
					$\sqcap$ & $\bot$ & $a$ & $\top$\\
					\hline
					$\bot$ & $\bot$ & $\bot$ & $\bot$\\
					$a$ & $\bot$ & $a$ & $a$\\
					$\top$ & $\bot$ & $a$ & $a$\\
					\hline
				\end{tabular} 
				\qquad
				\begin{tabular}{|c|ccc|}
					\hline 
					$\sqcup$ & $\bot$ & $a$ &$\top$\\ 
					\hline
					$\bot$ & $a$ & $a$ & $\top$\\
					$a$ & $a$ & $a$  & $\top$\\	
					$\top$ & $\top$ & $\top$ & $\top$ \\
					\hline
				\end{tabular} 
				\qquad
				\begin{tabular}{|c|ccc|}
					\hline 
					$x$& $\bot$ & $a$ & $\top$\\
					\hline
					$\neg x$ & $a$ & $\bot$ & $\bot$\\
					$\lrcorner x$ & $\top$ & $\top$ & $a$\\
					\hline
					
				\end{tabular}
				\caption{ Operations $\sqcap$, $\sqcup$, $\neg$ and $\lrcorner$ of $\underline{D}_{3,I}$}\label{OD3I}
			\end{figure}
			
			\item The algebra $\underline{D}_{3,II}:=(\{\bot, a, \top \}; \sqcap,\sqcup,\neg,\lrcorner,\bot,\top)$ given by the Cayley tables on 	Figure~\ref{OD3II}, is a \emph{trivial dBa which is neither pure nor regular}.
			Moreover,    $D_{3,II_\sqcap}=\{\bot, \top \}$ and $D_{3,II_\sqcup}=\{\top\}$.
			\begin{figure}[th]
				\begin{tabular}{|c|ccc|}
					\hline 
					$\sqcap$ & $\bot$ & $a$ & $\top$\\
					\hline
					$\bot$ & $\bot$ & $\bot$ & $\bot$\\
					$a$ & $\bot$ & $\bot$ & $\bot$\\
					$\top$ & $\bot$ & $\bot$ & $\top$\\
					\hline
				\end{tabular} 
				\qquad
				\begin{tabular}{|c|ccc|}
					\hline 
					$\sqcup$ & $\bot$ & $a$ &$\top$\\
					\hline
					$\bot$ & $\top$ & $\top$ & $\top$\\
					$a$ & $\top$ & $\top$  & $\top$\\	
					$\top$ & $\top$ & $\top$ & $\top$ \\
					\hline
				\end{tabular} 	
				\qquad
				\begin{tabular}{|c|ccc|}
					\hline 
					$x$& $\bot$ & $a$ & $\top$\\
					\hline
					$\neg x$ & $\top$ & $\top$ & $\bot$\\
					$\lrcorner x$ & $\top$ & $\top$ & $\top$\\
					\hline
				\end{tabular}
				\caption{Operations $\sqcap$, $\sqcup$, $\neg$ and $\lrcorner$ of $\underline{D}_{3,II}$}\label{OD3II}
			\end{figure}
			
			\item The algebra $\underline{D}_{6}:=(\{\bot, a, b, c, d, \top \}; \sqcap,\sqcup,\neg,\lrcorner,\bot,\top)$ given by the Cayley tables on Figure~\ref{OD6I}, is a \emph{pure and regular dBa which is not trivial}. Moreover, $D_{6_\sqcap}=\{\bot, a,b,c\}$ and  $D_{6_\sqcup}=\{a,c,d,\top\}$. Its Hasse diagram is given in Figure~\ref{D6}.
			
			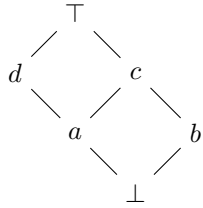
\begin{figure}[h]
				\begin{center}
					\begin{tikzpicture}[scale=.4]
						\node (one) at (0,2) {$\top$};
						\node (d) at (-2,0) {$d$};
						\node (c) at (2,0) {$c$};
						\node (a) at (0,-2) {$a$};
						\node (b) at (4,-2) {$b$};
						\node (zero) at (2,-4) {$\bot$};
						\draw (a)--(zero)--(b)--(c)--(a)--(d)--(one)--(c);
					\end{tikzpicture}
					\caption{A \emph{six}-element pure, regular and non trivial dBa}\label{D6}
				\end{center}
			\end{figure}
			
			
			\begin{figure}[ht]
				\begin{tabular}{|c|cccccc|}
					\hline 
					$\sqcap$ & $\bot$ & $a$ & $b$ &  $c$ & $d$ & $\top$\\
					\hline
					$\bot$ & $\bot$ & $\bot$ & $\bot$ & $\bot$ & $\bot$ & $\bot$\\
					$a$ & $\bot$ & $a$ & $\bot$ & $a$ & $a$ & $a$ \\
					$b$ & $\bot$ & $\bot$ & $b$ & $b$ & $\bot$ & $b$\\	
					$c$ & $\bot$ & $a$ & $b$ & $c$ & $a$ & $c$ \\
					$d$ & $\bot$ & $a$ & $\bot$ & $a$ & $a$ & $a$ \\
					$\top$ & $\bot$ & $a$ & $b$ & $c$ & $a$ & $c$\\
					\hline
				\end{tabular} 
				\qquad
				\begin{tabular}{|c|cccccc|}
					\hline 
					$\sqcup$ & $\bot$ & $a$ & $b$ & $c$ & $d$ & $\top$\\
					\hline
					$\bot$ & $a$ & $a$ & $c$ & $c$ & $d$ & $\top$ \\
					$a$ & $a$ & $a$ & $c$ & $c$ & $d$ & $\top$\\
					$b$ & $c$ & $c$ & $c$ & $c$ & $\top$ & $\top$ \\	
					$c$ & $c$ & $c$ & $c$ & $c$ & $\top$ & $\top$\\
					$d$ & $d$ & $d$ & $\top$ & $\top$ & $d$ & $\top$  \\
					$\top$ & $\top$ & $\top$ & $\top$ & $\top$ & $\top$ & $\top$\\
					\hline
				\end{tabular}
				
				\begin{tabular}{|c|cccccc|}
					\hline 
					$x$& $\bot$ & $a$ & $b$ & $c$ & $d$ & $\top$\\
					\hline
					$\neg x$ & $c$ & $b$ &  $a$ & $\bot$ & $b$ &$\bot$\\
					$\lrcorner x$ & $\top$ &  $\top$ & $d$ & $d$ & $c$ & $a$\\
					\hline
				\end{tabular}
				\caption{Operations $\sqcap$, $\sqcup$, $\neg$ and $\lrcorner$ of $\underline{D}_{6}$}\label{OD6I}
			\end{figure}
			
					
		\end{enumerate}
	\end{example}
	
	\newpage
	The following proposition is very useful when performing calculations in dBas.
	\begin{proposition}\label{p4}
		Let $\underline{D}$ be a dBa. For all $x,y\in D$, the following statements hold :
		\begin{enumerate}
			\item  $x\sqcap y\in D_{\sqcap} $ and $x\sqcup y\in D_{\sqcup}. $
			
			\item  $\neg x\in D_{\sqcap}$ and $\lrcorner x\in D_{\sqcup}.$
			
			\item $x\sqsubseteq y$ iff $\neg y\sqsubseteq \neg x$ and $\lrcorner y\sqsubseteq \lrcorner x.$ 
			
			\item $\neg\neg x=x\sqcap x$ and $\lrcorner\lrcorner x=x\sqcup x.$ 
			
			\item $x\vee y\in D_{\sqcap} $ and $x\wedge y\in D_{\sqcup}. $
			
			\item $\neg(x\vee y)=\neg x\sqcap \neg y$ and $\neg(x\sqcap y)=\neg x\vee \neg y.$
			
			\item $\lrcorner(x\wedge y)=\lrcorner x\sqcup \lrcorner y$ and $\lrcorner(x\sqcup y)=\lrcorner x\wedge \lrcorner y.$
		\end{enumerate}

	\end{proposition}
	
	Let $\underline{D}=(D;\sqcap,\sqcup,\neg,\lrcorner,\bot,\top)$ be a dBa. Recall that a congruence relation on $\underline{D}$ is an equivalence relation on $D$ that is compatible with $\sqcap$, $\sqcup$, $\neg$ and $\lrcorner$. i.e. $(a,b), (c,d)\in\theta \implies (a\sqcap c, b\sqcap d), (a\sqcup c, b\sqcup d), (\neg a, \neg b), (\lrcorner a, \lrcorner b)\in\theta$.

	\begin{notation}
		The set of all congruence relations on $\underline{D}$ is denoted by $Con(\underline{D})$. For any $\theta\in Con(\underline{D})$ and $a\in D$,  $[a]_{\theta}=\{\,x\in D\mid (a,x)\in\theta\, \}$ is called the $\theta$ class of $a$.
	\end{notation}
	
	\begin{example}
		The equivalence relation defined on $\underline{D}_{6}$ by its classes  $\{\bot, b \},\{a,c \}$ and $\{d,\top\}$ is a congruence relation.
	\end{example}
	
	Let  $\underline{D}=(D;\sqcap,\sqcup,\neg,\lrcorner,\bot,\top)$ be a dBa, we define two binary operations $\cdot$ and $+$ on $D$ by:
	$$x\cdot y:=(x\sqcup\lrcorner y)\wedge(\lrcorner x\sqcup y)\quad \text{ and }\quad x+y:=(x\sqcap\neg y)\vee(\neg x\sqcap y).$$
	These two operations are important when characterizing congruence relations on dBas.
	
	%
	Let $\underline{D}$ be a dba and $A$ a subset of $D$, we set $\neg A=\{\neg a: a\in A\}$ and $\lrcorner A=\{\lrcorner a: a\in A\}$.
	The following proposition allows in pure  dBa $\underline{D}$, to build a congruence relation from a pair $(I, F)$ where $I$ is an ideal of $\underline{D}_{\sqcap}$ and  $F$ is a filter of $\underline{D}_{\sqcup}$ such that $\neg F\subseteq I$ and $\lrcorner I\subseteq F$.
	\begin{proposition}\label{t8}
		Let $\underline{D}=(D;\sqcap,\sqcup,\neg,\lrcorner,\bot,\top)$ be a pure dBa. Let $I$ be an ideal of $\underline{D}_{\sqcap}$ and  $F$ be a filter of $\underline{D}_{\sqcup}$, such that $\neg F\subseteq I$ and $\lrcorner I\subseteq F$. Then the binary relation defined by : $$(a,b)\in\theta: \iff  a+ b\in I~\text{and}~ a\cdot b\in F$$ is a congruence relation on $\underline{D}.$
	\end{proposition}
	
	\begin{definition}
		In a dBa $\underline{D}$, a pair $(I,F)$ where $I$ is an ideal of $\underline{D}_{\sqcap}$, $F$ is a filter of $\underline{D}_{\sqcup}$ such that $\neg F\subseteq I$, $\lrcorner I\subseteq F$ is called \emph{a congruence generating pair}.
	\end{definition}
	
	\begin{notation} Let $\underline{D}$ be a dBa.
		For every congruence generating pair $(I,F)$, we denote by $\theta_{I,F}$ the congruence relation on $\underline{D}$ generated by $(I,F)$ and by $\mathfrak{C}(\underline{D})$ the set of all congruence generating pairs of $\underline{D}$.
	\end{notation}
	
	On the set $\mathfrak{C}(\underline{D})$ of all congruence generating pairs of a given dBa $\underline{D}$, the following  order relation is defined: $$(I,F)\le(G,H)\iff I\subseteq G~\text{and}~F\subseteq H.$$

	\begin{corollary}\cite[Corollary 2, p. 12]{11}\label{c6}
		If $\underline{D}$ is a pure double Boolean algebra, then the  map $\phi: Con(\underline{D})\longrightarrow \mathfrak{C}(\underline{D}),  \theta \longmapsto([\bot]_{\theta}\cap D_{\sqcap},~[\top]_{\theta}\cap D_{\sqcup})$
		is an isomorphism (with the inverse given by  Proposition \ref{t8}) between the lattice $Con(\underline{D})$ and the ordered set of all congruence generating pairs of $\underline{D}.$ 
	\end{corollary}
	
	The following proposition allows for a dBa $\underline{D}$, to extend a congruence on $\underline{D}_{p}$ to a congruence on $\underline{D}$.
	\begin{proposition}\cite[Theorem 7, p. 16]{11}\label{t9}
		If $\underline{D}$ is a double Boolean algebra and $\theta$ a congruence relation on $\underline{D}_{p}$ then $\theta':=\theta\cup \Delta_{D}$ is a congruence relation on $\underline{D}$.
	\end{proposition}
	
	Now, we can present our results. 
	
	\section{Simple double Boolean algebras }
	First we study some properties of dBas,  then show how to construct trivial double Boolean algebras from Boolean algebras and thereafter, characterize simple dBas. 
	
	\subsection{Some results on double Boolean algebras}
	
	The first result here characterizes the quasi-order on pure dBas.
	
	\begin{proposition}\label{p11} Every pure double Boolean algebra is regular.
	\end{proposition}
	\begin{proof}
		Assume that $\underline{D}$ is a pure dBa. Let $x,y\in D$ such that $x\sqsubseteq y$ and $y\sqsubseteq x$. We have:
		$$\big(x\sqsubseteq y ~\text{and}~ y\sqsubseteq x\big) \iff \left\{
		\begin{array}{r c l}
			x\sqcap x&=&y\sqcap y\\
			x\sqcup x&=&y\sqcup y
		\end{array}
		\right.~~(\text{by}~(3)~\text{of Proposition}~\ref{p8}).$$
		
		\begin{itemize}
			\item If $x,y\in D_{\sqcap}$ (resp. $x,y\in D_{\sqcup}$), then we are done, because $\sqsubseteq_{\sqcap}$ ( resp. $\sqsubseteq_{\sqcup}$) is an order relation on $D_{\sqcap}$ ( resp. $D_{\sqcup}$).
			\item If $x\in D_{\sqcap}$ and $y\in D_{\sqcup}$, then $$x=x\sqcap x = y\sqcap y~(\star) \quad \text{ and }\quad y= y\sqcup y = x\sqcup x~(\star\star).$$ 
			Therefore :
			$$\begin{array}{lll}
				y&=(y\sqcap y)\sqcup(y\sqcap y)& (\text{by}~ (\star)~\text{and}~(\star\star) ) \\
				&=(y\sqcup y)\sqcap(y\sqcup y)& (\text{by axiom}~(12))\\
				&=y\sqcap y& (\text{because}~y\in D_{\sqcup})\\
				&=x& (\text{by}~(\star))
			\end{array}$$
			\item If $x\in D_{\sqcup}$ and $y\in D_{\sqcap}$, then a similar calculation shows that $x=y$  and we are done.
		\end{itemize}
	\end{proof}
	
	\begin{remark}\label{r1}
		Let $\underline{D}$ be a dBa. The following statements hold:
		
		\begin{enumerate}
			\item  If $x\in D_{\sqcap}$, then $x\sqcup x\in D_{\sqcap}\cap D_{\sqcup}$.
			\item  If $x\in D_{\sqcup}$, then $x\sqcap x\in D_{\sqcap}\cap D_{\sqcup}$.
		\end{enumerate}
	\end{remark}
	The next result shows that in a trivial dBa $\underline{D}$, the operations $\sqcup$ and $\lrcorner$ (resp. $\sqcap$ and $\neg$) are constant on $D_{\sqcap}$ (resp. $D_{\sqcup}$). 
	\begin{proposition}\label{p9}
		Let $\underline{D}$ be a trivial dBa. The following statements hold:
		\begin{enumerate}
			\item  If $x,y\in D_{\sqcap}$, then $x\sqcup y=\bot\sqcup\bot$ and $\lrcorner x=\top$.
			\item  If $x,y\in D_{\sqcup}$, then $x\sqcap y=\top\sqcap\top$ and $\neg x=\bot$.
		\end{enumerate}
	\end{proposition}
	
	\begin{proof}
		
		The items $(1)$ and $(2)$ are dual. We give a proof of $(1)$.
		
		Since $\underline{D}$ is a trivial dBa,  $D_{\sqcap}\cap D_{\sqcup}=\{\top\sqcap\top\}=\{\bot\sqcup\bot \}~(\dagger)$.
		
		Let $x,y\in D_{\sqcap}$. We have :
		$$\begin{array}{ccl}
			x\sqcup y&=&(x\sqcup x)\sqcup(y\sqcup y)~~(\text{by axiom}~(1b))\\
			&=& (\bot\sqcup \bot)\sqcup (\bot\sqcup \bot)~~(\text{by Remark \ref{r1}}~\text{and}~(\dagger))\\
			&=& \bot\sqcup \bot~~(\text{by Proposition}~ \ref{p4})
		\end{array}$$
		and
		$$\begin{array}{ccl}
			\lrcorner x&=&\lrcorner(x\sqcup x)~~(\text{by axiom}~(8b))\\
			&=& \lrcorner(\bot\sqcup \bot)~~(\text{{by Remark \ref{r1}}}~\text{and}~(\dagger))\\
			&=& \lrcorner\bot~~(\text{by axiom}~(8b))\\
			&=& \top~~(\text{by axiom}~(11b)).
		\end{array}$$
		
	\end{proof}
	
	\begin{definition} \cite
		{5}
		Let $(P;\le_{P})$ and $(Q;\le_{Q})$ be two posets.
		
		\begin{enumerate}
			\item  The \emph{ordinal sum} $P+Q$ of $P$ and $Q$ is the poset $(P\cup Q; \le)$ where for elements $x,y\in P\cup Q$, $x\le y$ if one of the following conditions holds:
			
			\begin{enumerate}[label=\textup{(\alph*)}]
				\item  $x,y\in P$ and $x\le_{P} y$,
				\item  $x,y\in Q$ and $x\le_{Q} y$,
				\item  $x\in P$ and $y\in Q$.
			\end{enumerate}
			\item  If $(P;\le_{P})$ have a unit $1_{P}$ and $(Q;\le_{Q})$ have a zero $0_{Q}$ then the \emph{glued sum} $P\overset{\mathbf{\bullet}}{+} Q$ is obtained from $P+Q$ by identifying $1_{P}$ and $0_{Q}$; that is $1_{P}=0_{Q}$.
		\end{enumerate}
	\end{definition}
	
	Now, we can give a characterization of pure and trivial dBas. 
	
	\begin{theorem}\label{t13}
		
		\begin{enumerate}
			\item If $\underline{D}$ is a pure and trivial dBa, then $(D;\sqsubseteq)=\underline{D}_{\sqcap}\overset{\mathbf{\bullet}}{+} \underline{D}_{\sqcup}$.
			\item If $D=\underline{P}\overset{\mathbf{\bullet}}{+} \underline{Q}$ where $\underline{P}:=(P;\wedge_P, \vee_P,^{\prime_P}, 0_P, 1_P)$ and $\underline{Q}: =(Q;\wedge_Q, \vee_Q,^{\prime_Q}, 0_Q, 1_Q)$ are Boolean algebras, then $(D;\sqcap, \sqcup, \neg,\lrcorner,\bot,\top)$ is a pure and trivial dBa, where $\bot=0_{P}$, $\top=1_{Q}$,
			
			$x\sqcap y:=\left\{
			\begin{array}{cl}
				x\wedge_{P} y& \text{if}~x,y\in P,\\
				1_{P}=0_{Q}& \text{if}~x,y\in Q,\\
				x& \text{if}~x\in P~\text{and}~y\in Q
			\end{array}
			\right.$;
			
			$x\sqcup y:=\left\{
			\begin{array}{cl}
				x\vee_{Q} y& \text{if}~x,y\in Q,\\
				0_{Q}=1_{P}& \text{if}~x,y\in P,\\
				y& \text{if}~x\in P~\text{and}~y\in Q
			\end{array}
			\right.; $
			
			$\neg x:=\left\{
			\begin{array}{cl}
				x^{\prime_P}& \text{if}~x\in P,\\
				0_{P}& \text{otherwise}
			\end{array}
			\right.$;
			
			$\lrcorner x:=\left\{
			\begin{array}{cl}
				x^{\prime_Q}& \text{if}~x\in Q,\\
				1_{Q}& \text{otherwise}
			\end{array}
			\right.. $ 
		\end{enumerate}
		
	\end{theorem}
	
	\begin{proof}
		
		\begin{enumerate}
			\item Assume that $\underline{D}=(D;\sqcap,\sqcup,\neg,\lrcorner,\bot,\top)$ is a pure and trivial dBa.
			Set $(P;\le_{P})=(D_{\sqcap};\sqsubseteq_{\sqcap})$ and $(Q;\le_{Q})=(D_{\sqcup};\sqsubseteq_{\sqcup})$. Since $\underline{D}$ is pure and trivial, $D=P\cup Q$, $\sqsubseteq$ is an order relation and $1_{P}=\top\sqcap\top=\bot\sqcup\bot=0_{Q}$. Let $x,y\in D$.
			\begin{itemize}
				\item If $x,y\in P$, then $x\sqsubseteq y\iff x\sqsubseteq_{\sqcap} y$ (by $(1)$ of Proposition \ref{p8}). Thus, $x\le_{P} y$. 
				\item If $x,y\in Q$, then $x\sqsubseteq y\iff x\sqsubseteq_{\sqcup} y$ (by $(2)$ of Proposition \ref{p8}). Thus, $x\le_{Q} y$.
				\item If $x\in P$ and $y\in Q$, then 
				\begin{align*}
					x\sqcap y &= x\sqcap(y\sqcap y) ~~(\text{by axiom}~(1a)) \\
					&=	 x\sqcap(\top\sqcap\top)~~(\text{by Proposition}~\ref{p9})\\ 
					&=x \qquad \text{ and } \\
					x\sqcup y &= (x\sqcup x)\sqcup y~~(\text{by axiom}~(1b))\\
					&=(\bot\sqcup\bot)\sqcup y~~(\text{by Proposition}~\ref{p9})\\
					&=y
				\end{align*}
				Thus $x\sqsubseteq y$. 
			\end{itemize}	  
			Hence $D=P\overset{\bullet}{+}Q$.
			\item Assume that $D=P\overset{\bullet}{+}Q$ where $\underline{P}=(P; \wedge_{P},\vee_{P}, ^{\prime_P}, 0_{P}, 1_{P})$ and $\underline{Q}=(Q; \wedge_{Q},\vee_{Q}, ^{\prime_Q}, 0_{Q}, 1_{Q})$ are Boolean algebras.		
			We will show that the axioms $(1a)-(11a)$ and $(12)$ are satisfied. The axioms $(1b)-(11b)$ can be obtained dually.
			Let $x,y,z\in D$. Set $x\vee y:=\neg(\neg x\sqcap\neg y).$
			\begin{enumerate}[label=\textup{(\alph*)}]
				\item~ 
				\begin{enumerate} [label=\textup{(\roman*)}]
					\item If $x,y\in P,$ then $(x\sqcap x)\sqcap y=x\sqcap y.$
					
					\item  If $x,y\in Q,$ then $(x\sqcap x)\sqcap y=1_{P}\sqcap y=1_{P}=x\sqcap y.$
					
					\item  If $x\in P$ and $y\in Q$, then $(x\sqcap x)\sqcap y=x\sqcap y.$
					
					\item  If $x\in Q$ and $y\in P$, then $(x\sqcap x)\sqcap y=1_{P}\sqcap y=y=x\sqcap y.$
					
				\end{enumerate}
				In any case we have $(x\sqcap x)\sqcap y=x\sqcap y$, thus the axiom $(1a)$ holds.
				\item  Since a similar check as in (a) can give also the equality $x\sqcap(y\sqcap y)=y$ for $x,y\in D$, we have : $$\begin{array}{lcl}
					x\sqcap y&=&(x\sqcap x)\sqcap(y\sqcap y)~~(\text{by}~(a))\\
					&=&(x\sqcap x)\wedge_{P}(y\sqcap y) \left( \text{ due to }\left( x\sqcap x\right)\in P\right) \\
					&=&(y\sqcap y)\wedge_{P}(x\sqcap x)~~(\text{because}~\wedge_{P}~\text{is commutative in }~P).\\
					&=&(y\sqcap y)\sqcap(x\sqcap x)\\
					&=& y\sqcap x~~~(\text{by}~(a)).
				\end{array}$$
				
				Thus the axiom $(2a)$ holds.
				
				\item The axiom $(3a)$ follows from $(a)$ and the  associativity of $\wedge_{P}$. 
				\item~
				\begin{enumerate}[label=\textup{(\roman*)}]
					\item  If $x,y\in P,$ then $x\sqcap(x\sqcup y)=x\sqcap 1_{P}=x=x\sqcap x.$
					\item  If $x,y\in Q,$ then $x\sqcap(x\sqcup y)=1_{P}=x\sqcap x.$
					\item If $x\in P$ and $y\in Q$, then $x\sqcap(x\sqcup y)=x\sqcap y=x=x\sqcap x.$
					\item If $x\in Q$ and $y\in P$, then $x\sqcap(x\sqcup y)=x\sqcap x.$
				\end{enumerate}
				
				In any case we have $x\sqcap(x\sqcup y)=x\sqcap x$; so the axiom $(4a)$ holds.
				\item~  
				\begin{enumerate}[label=\textup{(\roman*)}]
					\item [(i)]If $x\in P$, then $\neg(x\sqcap x)=\neg x.$ 
					\item [(ii)]If $x\in Q$, then $\neg (x\sqcap x)= 1_{P}^{\prime_P}=0_{P}=\neg x.$ In any case we have $\neg(x\sqcap x)=\neg x$; so the axiom $(8a)$ holds.
				\end{enumerate}
				\item  We have $$\begin{array}{lcl}
					x\vee y&=&\neg(\neg x\sqcap\neg y)\qquad \text{ by definition} \\
					&=&\neg(\neg(x\sqcap x)\sqcap\neg(y\sqcap y)) ~~(\text{by}~(e))\\
					&=& (x\sqcap x)\vee (y\sqcap y) \qquad \text{ by definition}.
				\end{array}$$
				
				\item  Using $(a)$, $(f)$ and the absorption law in $\underline{P}$, we have $x\sqcap(x\vee y)=x\sqcap x$; so the axiom $(5a)$ holds.
				\item  Using $(a)$, $(f)$ and the distributive law in $\underline{P}$, we have $x\sqcap(y\vee z)=(x\sqcap y)\vee(x\sqcap z)$; so the axiom $(6a)$ holds.
				\item   $\neg\neg(x\sqcap y) = (x\sqcap y)^{\prime_P\prime_P} = x\sqcap y$~(because $\underline{P}$ is a Boolean algebra); so the axiom $(7a)$ holds.
				\item  $$\begin{array}{lcl}
					x\sqcap \neg x&=& (x\sqcap x)\sqcap\neg(x\sqcap x)~~(\text{by}~(a)~\text{and}~(e))\\
					&=& (x\sqcap x)\wedge_{P} (x\sqcap x)^{\prime_P}\\
					&=& 0_{P}~~(\text{because}~\underline{P}~\text{is a Boolean algebra})\\
					&=&\bot.
				\end{array}$$
				
				So the axiom $(9a)$ holds.
				\item   $\neg\bot=\neg 0_{P} = 0_{P}^{\prime_P} = 1_{P} = 1_{Q} \sqcap 1_{Q}=\top\sqcap\top$. So the axiom $(10a)$ holds.
				\item  $\neg \top=\neg 1_{Q}:=\bot$. So the axiom $(11a)$ holds.
				\item  $(x\sqcap x)\sqcup(x\sqcap x)=1_{P}$ (because $x\sqcap x\in P$) and  $(x\sqcup x)\sqcap(x\sqcup x)=1_{P}$ (because $x\sqcup x\in Q$) so $(x\sqcap x)\sqcup(x\sqcap x)=(x\sqcup x)\sqcap(x\sqcup x)$; hence, the axiom $(12)$ holds.
			\end{enumerate}
			
			Thus $\underline{D}=(D;\sqcup,\sqcap,\neg,\lrcorner,\bot,\top)$ is a dBa. Moreover, $D= P\cup Q=D_{\sqcap}\cup D_{\sqcup}$ and 
			$\bot\sqcup\bot=0_{P}\sqcup 0_{P}=0_{Q}=1_{P}~(\text{because}~D=P\overset{\bullet}{+}Q)=1_{Q}\sqcap 1_{Q}=\top\sqcap\top.$
			So $\underline{D}=(D;\sqcup,\sqcap,\neg,\lrcorner,\bot,\top)$ is pure and trivial.
			We conclude that $\underline{D}=(D;\sqcap,\sqcup,\neg,\lrcorner,\bot,\top)$ is a pure and trivial dBa.
		\end{enumerate}
	\end{proof}
	
	
	The next corollary is a direct application of Theorem \ref{t13}.
	\begin{corollary}
		Let $A$ and $B$ be two disjoint sets. Then there exists a structure of pure and trivial  dBa on the set $P(A) \cup P(B)$. 
	\end{corollary}

	In the following proposition we give some properties of the operations $"+"$ and $"\cdot"$ important for calculations on congruence relations.
	
	\begin{proposition}\label{p1} Let $\underline{D}$ be a dBa, $a,b\in D$, $I$ be an ideal of $\underline{D}_{\sqcap}$ and $F$ a filter of $\underline{D}_{\sqcup}$. The following statements hold :
		
		\begin{enumerate}
			\item  $a+b\in D_{\sqcap}$ and $a\cdot b\in D_{\sqcup}$.
			\item  $a+ b=b+a$ ~~and~~$a\cdot b=b\cdot a$.
			\item  $a+ a=\bot$ ~~and~~$a\cdot a=\top$.
			\item  $a+ \bot=a\sqcap a$~~and~~$a\cdot \top=a\sqcup a$.
			\item  $a+\top=\neg a$~~and~~$a\cdot\bot=\lrcorner a$. 		
			\item  $a,b\in I\implies a+b\in I$~~and~~$a,b\in F\implies a\cdot b\in F$.
			\item  $a+(b\sqcap b)=a+b$~~and~~$a\cdot (b\sqcup b)=a\cdot b$.
			\item  $(a\sqcap a)+(b\sqcap b)=a+b$~~and~~$(a\sqcup a)\cdot (b\sqcup b)=a\cdot b$.
			\item  $(a+b)+c=a+(b+c)$~~and~~$(a\cdot b)\cdot c=a\cdot(b\cdot c)$.
			\item  For a regular dBa $\underline{D}$, $(a+b=\bot~\text{and}~ a\cdot b=\top) \iff a=b$.
			
		\end{enumerate}
		
	\end{proposition}
	
	\begin{proof}
		
		\begin{enumerate}
			\item  Follows from $(5)$ of  Proposition \ref{p4}.
			\item  Follows from the commutativity of "$\sqcap$" and "$\sqcup$". 
			\item  We have $a+ a:=(a\sqcap \neg a)\vee(\neg a\sqcap a)$; using axiom $(9a)$, we obtain $a+a=\bot\vee\bot=\bot.$ Analogously we have $a\cdot a=\top$.
			\item  We have $a+ \bot:=(a\sqcap \neg \bot)\vee(\neg a\sqcap \bot)=(a\sqcap\top\sqcap\top)\vee\bot~(\text{by axiom}~(10a))=(a\sqcap a)\vee\bot=a\sqcap a.$
			Similarly, we have $a\cdot \top=a\sqcup a$.
			\item  We have $a+ \top:=(a\sqcap \neg \top)\vee(\neg a\sqcap \top)=(a\sqcap\bot)\vee(\neg a\sqcap\neg a)~(\text{by axiom}~(11a))=\bot\vee\neg a=\neg a.$ Similarly, we have $a\cdot \bot=\lrcorner a$.
			
			\item  Obvious.
			
			\item  We have $a+ (b\sqcap b):=(a\sqcap \neg (b\sqcap b))\vee(\neg a\sqcap b\sqcap b)=(a\sqcap\neg b)\vee(\neg a\sqcap b)~ (\text{by axioms}~(1a),~(8a)):=a+b.$ Similarly, we have $a\cdot (b\sqcup b)=a\cdot b$.
			
			\item  Consequence of (7).
			
			\item We have:   
			
			$$\begin{array}{lcll}
				(a+ b)+ c&:=&(a\sqcap a+b\sqcap b)+c\sqcap c)&(\text{by}~(8))\\
				&=& a\sqcap a+(b\sqcap b+c\sqcap c)& (\text{because }~+~\text{is associative in}~D_{\sqcap}) \\
				&=& a+(b+c)& (\text{by}~(8)).
				
			\end{array}$$
			Similarly, we have $(a\cdot b)\cdot c=a\cdot(b\cdot c)$.
			
			\item Assume that $\underline{D}$ is a regular dBa. Then we have:
			
			$$\begin{array}{lll}
				\left\{\begin{array}{lcl}
					a+b&=&\bot\\
					a\cdot b&=&\top
				\end{array}
				\right.&\implies& \left\{\begin{array}{lcl}
					a+b+b&=&\bot+b\\
					a\cdot b\cdot b&=&\top\cdot b
				\end{array}
				\right.\\
				
				&\implies& \left\{\begin{array}{lcl}
					a+\bot&=&\bot+b\\
					a\cdot\top&=&\top\cdot b
				\end{array}
				\right.~~(\text{by}~(3))\\
				&\implies& \left\{\begin{array}{lcl}
					a\sqcap a&=&b\sqcap b\\
					a\sqcup a&=&b\sqcup b
				\end{array}
				\right.~~(\text{by}~(4))\\
				&\implies& \left\{\begin{array}{lcl}
					a\sqcap a\sqsubseteq_{\sqcap} b\sqcap b&\text{and}&b\sqcap b\sqsubseteq_{\sqcap} a\sqcap a\\
					a\sqcup a\sqsubseteq_{\sqcup} b\sqcup b&\text{and}&b\sqcup b\sqsubseteq_{\sqcup} a\sqcup a
				\end{array}
				\right.\\
				&\implies& a\sqsubseteq b~\text{and}~ b\sqsubseteq a~~(\text{by}~(3)~\text{of Proposition}~\ref{p8})\\
				
				&\implies& a=b~~(\text{because}~\underline{D}~\text{is regular}).
			\end{array}$$
		\end{enumerate}
		The converse is obvious.
		
	\end{proof}
	
	The following results put more light on the description of congruence relations of a pure and trivial dBa and the distributivity of the lattice of those congruence relations.
	
	\begin{corollary}\label{c5} Let $\underline{D}$ be a pure and trivial dBa.
		
		\begin{enumerate}
			\item If $I$ is an ideal of $\underline{D}_{\sqcap}$ and  $F$ is a filter of $\underline{D}_{\sqcup}$, then $$(a,b)\in\theta\iff  a+ b\in I~\text{and}~ a\cdot b\in F$$ defines a congruence relation on $\underline{D}.$
			\item The map 
			\[\begin{array}{clcc}
				\phi:& Con(\underline{D})&\longrightarrow& \mathcal{I}(\underline{D}_{\sqcap})\times \mathcal{F}(\underline{D}_{\sqcup})\\
				&       & &\\
				& \theta& \longmapsto&([\bot]_{\theta}\cap D_{\sqcap},~[\top]_{\theta}\cap D_{\sqcup})
			\end{array}\] is an isomorphism (with the inverse given by $(1)$) between the lattice of congruence relations on $\underline{D}$ and the lattice $\mathcal{I}(\underline{D}_{\sqcap})\times \mathcal{F}(\underline{D}_{\sqcup})$.
		\end{enumerate}
		
	\end{corollary}
	
	\begin{proof}
		
		\begin{enumerate}
			\item Let $I$ be an ideal of $\underline{D}_{\sqcap}$ and let $F$ be a filter of $\underline{D}_{\sqcup}$. By   Proposition \ref{p9} we have $\neg F=\{\bot \}\subseteq I$ and $\lrcorner I=\{\top \}\subseteq F$. Therefore applying Proposition \ref{t8} we obtain the result.
			\item By  $(1)$, $\mathfrak{C}(\underline{D})=\mathcal{I}(\underline{D}_{\sqcap})\times \mathcal{F}(\underline{D}_{\sqcup})$ and by applying Corollary \ref{c6} we obtain the result. 
		\end{enumerate}
		
	\end{proof}
	
	\begin{corollary}\label{c3}
		Let $\underline{D}$ be a pure and trivial dBa. The lattices $Con(\underline{D})$ and $Con(\underline{D}_{\sqcap})\times Con(\underline{D}_{\sqcup})$ are isomorphic.
	\end{corollary}
	
	\begin{proof}
		By $(2)$ of Corollary \ref{c5}, we have $Con(\underline{D})\cong \mathcal{I}(\underline{D}_{\sqcap})\times \mathcal{F}(\underline{D}_{\sqcup})$. Since $\mathcal{I}(\underline{D}_{\sqcap})\cong Con(\underline{D}_{\sqcap})$ and $\mathcal{F}(\underline{D}_{\sqcup})\cong Con(\underline{D}_{\sqcup})$ (by Proposition \ref{p10}), we have $Con(\underline{D})\cong Con(\underline{D}_{\sqcap})\times Con(\underline{D}_{\sqcup})$.
	\end{proof}
	
	\begin{corollary}
		The class of pure and trivial dBas is congruence-distributive.
	\end{corollary}
	
	\begin{proof}
		Let $\underline{D}$ be a pure and trivial dBa. Then by Corollary \ref{c3}, $Con(\underline{D})$ and $Con(\underline{D}_{\sqcap})\times Con(\underline{D}_{\sqcup})$ are isomorphic; moreover $Con(\underline{D}_{\sqcap})$ and $Con(\underline{D}_{\sqcup})$ are distributive lattices ( by Proposition \ref{p10}), hence  $Con(\underline{D})$ is distributive as a direct product of two distributive lattices.
	\end{proof}
	
	In order to facilitate the description of sub-directly irreducible double Boolean algebra, we divide the class of dBas into five sub-classes. 
	
	\begin{definition}
		Let $\underline{D}$ be a dBa.
		\begin{enumerate}
			\item  $\underline{D}$ is of \emph{type $I$} if 
			$D_{\sqcup}=\{\top \}$.
			\item  $\underline{D}$ is of \emph{type $II$} if $D_{\sqcap}=\{\bot \}$. 
			\item  $\underline{D}$ is of \emph{type $III$} if $\bot\sqcup\bot=\bot$ and $\top\sqcap\top=\top$.
			\item  $\underline{D}$ is of \emph{type $IV$} if $\bot=\top$ or $\bot\sqcup\bot\neq\bot$ or $\top\sqcap\top\neq\top$.
			\item  $\underline{D}$ is of \emph{type $V$} if $\underline{D}_{p}$ is a Boolean algebra.
		\end{enumerate}
	\end{definition}
	
	\begin{remark}
		
		\begin{enumerate}
			\item 	Every dBa of type $I$ (resp. type $II$) is a trivial dBa, every trivial dBa is a dBa of type $IV$ and every dBa of type $V$ is a dBa of type $III$.
			\item Every dBa is either of type $III$ or type $IV$.
		\end{enumerate}
		
	\end{remark}
	
	Before continuing, we give an example to clarify these types.
	
	\begin{example}\label{e2}\text{}
		
		\begin{enumerate}
			\item  Let $\underline{B}=(B;\wedge, \vee,~', 0, 1)$ be a Boolean algebra. The algebra 
			
			$\underline{D}=(D;\sqcap,\sqcup,\neg,\lrcorner,\bot,\top)$ where $D=B$, $\bot=0$, $\top=1$ and for all $x, y\in D$,  $x\sqcup y=\top$, $x\sqcap y=x\wedge y$, $\neg x=x'$, $\lrcorner x=\top$ (resp. $x\sqcap y=\bot$, $x\sqcup y=x\vee y$, $\neg x=\bot$, $\lrcorner x=x'$) is a \emph{dBa of type $I$} (resp. \emph{ type $II$}). In particular, if $B$ is infinite, then $D$ is also infinite.
			\item  The dBa $\underline{D}_{6}$ of Example \ref{e1} is of \emph{type $IV$.}
			\item   The algebra  $\underline{D}_{4}=(\{\bot, a, b,  \top \};\sqcap,\sqcup,\neg,\lrcorner,\bot,\top)$ with $D_{4_\sqcap}=\{\bot,\top \}$, $D_{4_\sqcup}=\{\bot, a, b, \top\}$, the Hasse diagram given in  Figure~\ref{D4}  and the Cayley tabular given in  Table~\ref{OD4} is a pure dBa of type $III$.
			
			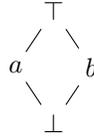
\begin{figure}[h]
				\begin{center}
					\begin{tikzpicture}[scale=.5]
						\node (one) at (0,1.5) {$\top$};
						\node (a) at (-1,0) {$a$};
						\node (d) at (1,0) {$b$};
						\node (zero) at (0,-1.5) {$\bot$};
						\draw (zero)--(a)--(one)--(d)--(zero);
					\end{tikzpicture}
					\caption{\emph{four}-element  dBa of type $III$}\label{D4}
				\end{center}
			\end{figure}
			
			\begin{figure}[h]
				
				\begin{tabular}{|c|cccc|}
					\hline 
					$\sqcap$ & $\bot$ & $a$ & $b$ &  $\top$\\
					\hline
					$\bot$ & $\bot$ & $\bot$ & $\bot$ & $\bot$ \\
					$a$ & $\bot$ & $\bot$ & $\bot$ & $\bot$ \\
					$b$ & $\bot$ & $\bot$ & $\bot$ & $\bot$ \\	
					$\top$ & $\bot$ & $\bot$ & $\bot$ & $\top$ \\
					\hline
				\end{tabular} 
				\qquad
				\begin{tabular}{|c|cccc|}
					\hline 
					$\sqcup$ & $\bot$ & $a$ & $b$ &$\top$\\
					\hline
					$\bot$ & $\bot$ & $a$ & $b$ &$\top$\\
					$a$ & $a$ & $a$ & $\top$ &$\top$\\
					$b$ & $b$ & $\top$ & $b$ & $\top$\\
					$\top$ & $\top$ & $\top$ & $\top$ & $\top$\\
					\hline
				\end{tabular}	
				\qquad
				\begin{tabular}{|c|cccc|}
					\hline 
					$x$& $\bot$ & $a$ & $b$ & $\top$\\
					\hline
					$\neg x$ & $\top$ & $\top$ & $\top$ & $\bot$\\
					$\lrcorner x$ & $\top$ & $b$ & $a$ & $\bot$ \\
					\hline	
				\end{tabular}
				\caption{ Operations $\sqcap$, $\sqcup$, $\neg$ and $\lrcorner$ of $\underline{D}_{4}$}\label{OD4}
			\end{figure} 
		\end{enumerate}
	\end{example}
	
	As algebra of type $I$ or $II$ are just light modifications of Boolean algebras, their lattice of congruence relations is easy to obtain as specify by the following corollary.
	
	\begin{corollary}\label{c1}
		If $\underline{D}$ is a pure dBa of type $I$ (resp. type $II$), then the lattices $Con(\underline{D})$ and $Con(\underline{D}_{\sqcap})$ (resp. $Con(\underline{D}_{\sqcup})$ ) are isomorphic. 
	\end{corollary}
	
	\begin{proof}
		Let $\underline{D}$ be a pure dBa of type $I$ (resp. type $II$). Then  $\underline{D}$ is pure and trivial. By Corollary \ref{c3}, the lattices $Con(\underline{D})$ and $Con(\underline{D}_{\sqcap})\times Con(\underline{D}_{\sqcup})$ are isomorphic. Since $D_{\sqcup}=\{\top \}$ (resp. $D_{\sqcap}=\{\bot \}$ ) we have $ Con(\underline{D}_{\sqcup})=\{\Delta_{D_{\sqcup}} \}$ (resp. $ Con(\underline{D}_{\sqcap})=\{\Delta_{D_{\sqcap}} \}$ ). Moreover, $Con(\underline{D}_{\sqcap})\times \{\Delta_{D_{\sqcup}} \}\cong Con(\underline{D}_{\sqcap}) $ (resp. $ \{\Delta_{D_{\sqcap}} \}\times Con(\underline{D}_{\sqcup})\cong Con(\underline{D}_{\sqcup}) $). 

	\end{proof}
	
	We are ready to characterize simple dBas.
	
	\subsection{Simple double Boolean algebras}
	
	We start with non pure and then pure simple dBas.
	
	\begin{proposition}\label{p23}
		Let $\underline{D}$ be a dBa that is not pure. Then,
		$\underline{D}$ is simple if and only if $|D|=2$.
	\end{proposition}
	
	\begin{proof}
		$\implies)$ We assume that $\underline{D}$ is a simple dBa. Since $D$ is not pure and $|D|>1$, there is $a\in D\setminus D_{p}$ .  The equivalence relation $\beta=\left( D\setminus\{a\}\right)^{2}\cup\{(a,a)\}$ is a congruence relation different from $D^{2}$. Hence $\beta=\Delta_{D}$, and the cardinality of $D$ is 2.
		
		$\impliedby)$ Obvious.
		
	\end{proof}
	
	
	\begin{proposition}\label{p22}
		Let $\underline{D}$ be a pure dBa. The following properties are equivalent:
		\begin{enumerate}
			\item $\underline{D}$ is simple
			\item ($D_{\sqcap}=\{\bot \}$ and $|D_{\sqcup}|\le 2$) or ($\top\sqcap\top=\top$ and the only ideals $J$ of $\underline{D}_{\sqcap}$ such that $\neg\lrcorner J\subseteq J$ are $\{\bot \}$ and $D_{\sqcap}$)
			\item ($D_{\sqcup}=\{\top \}$ and $|D_{\sqcap}|\le 2$) or ($\bot\sqcup\bot=\bot$ and the only filters $F$ of $\underline{D}_{\sqcup}$ such that $\lrcorner\neg F\subseteq F$ are $\{\bot \}$ and $D_{\sqcap}$)
		\end{enumerate}
	\end{proposition}
	
	\begin{proof}
		$(1)\implies (2)$ 
		
		If $D_{\sqcap}=\{\bot \}$, then let $a\in D_{\sqcup}$, $(\{\bot \}, F(a))\in\mathfrak{C}(\underline{D})$. Since $\underline{D}$ is simple, we have $F(a)=\{\top \}$ or $F(a)=D_{\sqcup}$; i.e. $a=\top$ or $a=\bot\sqcup\bot$. Hence $|D_{\sqcup}|\le 2$.
		We now assume that $D_{\sqcap}\neq\{\bot \}$, then $(\{\bot \}, F(\top\sqcap\top))\in\mathfrak{C}(\underline{D})$.  
		Thus $(\{\bot \},F(\top\sqcap\top))=(\{\bot\},\{\top \})$ or $(\{\bot \},F(\top\sqcap\top))=(D_{\sqcap},D_{\sqcup})$ (because $\underline{D}$ is simple). But $D_{\sqcap}\neq\{\bot \}$, thus  $(\{\bot \},F(\top\sqcap\top))=(\{\bot\},\{\top \})$. Therefore, $F(\top\sqcap\top)=\{\top\}$, and $\top\sqcap\top=\top$. Now,  let $J$ be an ideal of $\underline{D}_{\sqcap}$ such that $\neg\lrcorner J\subseteq J$, then $(J, F(\lrcorner J))\in\mathfrak{C}(\underline{D})$. Since $\underline{D}$ is simple, we have $(J, F(\lrcorner J))=(\{\bot \},\{\top \})$ or $(J, F(\lrcorner J))=(D_{\sqcap}, D_{\sqcup})$, that is $J=\{\bot \}$ or $J=D_{\sqcap}$.\\
		$(2)\implies (3)$
		
		Suppose that $D_{\sqcup}=\{\top\}$, if $D_{\sqcap}=\{\bot\}$ then we are done. Else, then for $a\in D_{\sqcap}$, $J=I(a)$ is an ideal of $\underline{D}_{\sqcap}$ such that $\neg\lrcorner J\subseteq J$.  By $(2)$, $I(a)=\{\bot \}$ or $I(a)=D_{\sqcap}$; i.e. $a=\bot$ or $a=\top\sqcap\top$. Hence, $|D_{\sqcap}|\le 2$.
		
		Now, we assume that, $D_{\sqcup}\neq\{\top\}$. If $D_{\sqcap}=\{\bot\}$, then $\bot\sqcup\bot=\bot$. If not, then $J=I(\bot\sqcup\bot)$ is an ideal of $\underline{D}_{\sqcap}$ such that $\neg\lrcorner J\subseteq J$ (due to $\neg\lrcorner x=\bot$ for all $x\sqsubseteq \bot\sqcup\bot$). 
		By $(2)$, $J=\{\bot \}$ or $J=D_{\sqcap}$. If $J= D_{\sqcap}$, then $\top\sqcap\top\sqsubseteq\bot\sqcup\bot$, and $\top=\top\sqcap\top\sqsubseteq\bot\sqcup\bot$, contradicting the assumption. Therefore, $J=\{\bot \}$, and $\bot\sqcup\bot=\bot$.
		Let $F$ be a filter of $\underline{D}_{\sqcup}$ such that $\lrcorner\neg F\subseteq F$. If $|D_{\sqcup}|\le 2$, then $F=\{\top \}$ or $F=D_{\sqcup}$.
		We assume that $|D_{\sqcup}|>2$. Then $\top\sqcap\top=\top$ (by $(2)$). We will show that $\neg\lrcorner I(\neg F)\subseteq I(\neg F)$. For this fact , let $x\in\neg\lrcorner I(\neg F)$, then there is $a\in I(\lrcorner F)$ such that $x=\neg\lrcorner a$. 
		Since $a$ is in the ideal generated by $\lrcorner F$, there are $f_{1},\ldots f_{n}\in F$ such that $a\sqsubseteq \neg f_{1}\vee\ldots\vee\neg f_{n}$. Thus $\neg\lrcorner a\sqsubseteq\neg\lrcorner \left( \neg f_{1}\vee\ldots\vee\neg f_{n}\right) 
		=\neg\lrcorner\left( \neg f_{1}\sqcup\ldots\sqcup\neg f_{n}\right)=\neg\left( \lrcorner\neg f_{1}\wedge\ldots\wedge\lrcorner\neg f_{n}\right)\in\neg F$. Therefore, $x\in I(\lrcorner F)$. Using again $(2)$, we get $I(\neg F)=\{\bot \}$ or $I(\neg F)=D_{\sqcap}$. Finally, an easy check gives $F=\{\top \}$ or $F=D_{\sqcup}$.\\
		$(3)\implies (1)$
		
		If $D_{\sqcap}=\{\bot\}$ and $|D_{\sqcap}|\le 2$, then $\underline{D}$ is simple.
		We assume now that $D_{\sqcap}\neq\{\bot\}$ or $|D_{\sqcap}|> 2$.
		Let $(J, F)\in\mathfrak{C}(\underline{D})$. Since $\neg\lrcorner F\subseteq F$, using $(3)$, we obtain $F=\{\top \}$ or $F=D_{\sqcup}$ (due to $\bot\sqcup\bot=\bot$). This implies $J=\{\bot \}$ or $J=D_{\sqcap}$. Therefore $(J,F)=(\{\bot\}, \{\top\})$ or $(J, F)=(D_{\sqcap}, D_{\sqcup})$ (due to $\top\sqcap\top\neq\bot=\bot\sqcup\bot$). Hence $\underline{D}$ is simple. 
	\end{proof}
	
		
	
	
	In the case of 
	finite dBas, filters and ideals are principal.   
	We get:
	\begin{corollary}\label{c11}
		Let $\underline{D}$ be a 
		finite pure dBa. The following properties are equivalent:
		\begin{itemize}
			\item [(a)] $\underline{D}$ is simple.
			\item [(b)] ($D_{\sqcap}=\{\bot \}$ and $|D_{\sqcup}|\le 2$) or ($\top\sqcap\top=\top$ and the only elements $a$ of $\underline{D}_{\sqcap}$ such that $\neg\lrcorner a\sqsubseteq a$ are $\bot$ and $\top$).
			\item [(c)] ($D_{\sqcup}=\{\top \}$ and $|D_{\sqcap}|\le 2$) or ($\bot\sqcup\bot=\bot$ and the only elements $b$ of $\underline{D}_{\sqcup}$ such that $b\sqsubseteq\lrcorner\neg b$ are $\bot $ and $\top$).
		\end{itemize}
	\end{corollary}
	
	\begin{proof}
		We apply Proposition \ref{p22}, and use the fact that 
		every ideal (resp. filter) of $\underline{D}_{\sqcap}$ (resp. $\underline{D}_{\sqcup}$ ) is principal. Thus ($\neg\lrcorner I(a)\subseteq I(a)\iff \neg\lrcorner a\sqsubseteq a$ for any $a\in D_{\sqcap}$) (resp. $\lrcorner\neg F(b)\subseteq F(b)\iff b\sqsubseteq\lrcorner\neg b$ for any $b\in D_{\sqcup}$). 
	\end{proof}
	
	\begin{example}\label{e6}
		
		\begin{enumerate}
			\item The double Boolean algebra $\underline{D}_{4}$ of Example \ref{e2} is simple. In fact, $\top\sqcap~\top=\top$ and $D_{\sqcap}=\{\bot,\top \}$. Thus the only elements $a$ of $\underline{D}_{\sqcap}$ such that $\neg\lrcorner a\sqsubseteq a$ are $\bot$ and $\top$. 
			\item The double Boolean algebra $\underline{D}_{6}$ of Example \ref{e1} is not simple because $|D_{\sqcup}|>2$ and $\top\sqcap\top\neq\top$.
		\end{enumerate}
		
	\end{example}
	The simple algebras of some sub-classes of dBas have  low cardinality as specified in the next corollary.
	\begin{corollary}\label{c10}
		Let $\underline{D}$ be a  dBa. If $\underline{D}$ is  a dBa of type $l$, $l\in\{I,II,IV,V \}$, then
		$\underline{D}$ is simple if and only if $|D|\le 2$.
	\end{corollary}
	There are also simple dBas of large cardinality in the subclass of algebras of type $III$. For example, we have $\underline{D}_{4}$ of Example \ref{e6}.
	
	Simple algebras are particular case of sub-directly irreducible algebras. We continue our contribution with the exploration of sub-directly irreducible members of sub-classes of type $l$, $l\in\{I, II, IV, V\}$.
	\section{Sub-directly irreducible  double Boolean algebras}
	
	In this section, we  determine up to isomorphism all sub-directly irreducible dBas of type $I$, $II$, $IV$ and  $V$.  
	
	\begin{proposition}\label{p7}
		There are exactly (up to isomorphism) four 
		\textit{two}-element double Boolean algebras.
	\end{proposition}
	
	\begin{proof}
		Let $\underline{D}_{2}$ be a \textit{two}-element double Boolean algebra, then
		
		$(\top\sqcap\top,\bot\sqcup\bot)\in\{(\bot,\bot),(\bot,\top),(\top,\bot),(\top,\top)\}$. Based on this set, we will  distinguish four cases :
		\begin{itemize}
			\item \textbf{\underline{Case 1}}:  $(\top\sqcap\top,\bot\sqcup\bot)=(\top,\top)$, then $\underline{D}_{2}$ is denoted by $\underline{D}_{2,I}$. 
			Its Cayley's tables are given 
			in Figure~\ref{OD2I}. Moreover, $D_{2,I_\sqcap}=\{\bot,\top\}$ and $D_{2,I_\sqcup}=\{\top\}$.
			
			\begin{figure}[h]
				\begin{tabular}{|c|cc|}
					\hline 
					$\sqcap$ & $\bot$ & $\top$\\
					\hline
					$\bot$ & $\bot$ & $\bot$  \\   
					$\top$ & $\bot$ & $\top$ \\
					\hline
				\end{tabular}
				\qquad
				\begin{tabular}{|c|cc|}
					\hline 
					$\sqcup$ & $\bot$ & $\top$\\
					\hline
					$\bot$ & $\top$ & $\top$  \\
					$\top$ & $\top$ & $\top$ \\
					\hline
					
				\end{tabular} 
				\qquad
				\begin{tabular}{|ccc|}
					\hline 
					$x$& $\bot$ &  $\top$\\
					\hline
					$\neg x$ & $\top$ & $\bot$\\
					$\lrcorner x$ & $\top$ & $\top$\\
					\hline
				\end{tabular}
				\caption{ Operations $\sqcap$, $\sqcup$, $\neg$ and $\lrcorner$ of $\underline{D}_{2,I}$}\label{OD2I}
			\end{figure}
			
			\item \textbf{\underline{Case 2}}: $(\top\sqcap\top,\bot\sqcup\bot)=(\bot,\bot)$, then $\underline{D}_{2}$ is denoted by $\underline{D}_{2,II}$. 
			Its Cayley's tables 
			are in Figure~\ref{OD2II}. Moreover, $D_{2,II_\sqcap}=\{\bot\}$ and $D_{2,II_\sqcup}=\{\bot, \top\}$.
			
			\begin{figure}[h]
				\begin{tabular}{|c|cc|}
					\hline 
					$\sqcap$ & $\bot$ & $\top$\\
					\hline
					$\bot$ & $\bot$ & $\bot$  \\
					$\top$ & $\bot$ & $\bot$ \\
					\hline
				\end{tabular}
				\qquad
				\begin{tabular}{|c|cc|}
					\hline 
					$\sqcup$ & $\bot$ & $\top$\\
					\hline
					$\bot$ & $\bot$ & $\top$  \\
					$\top$ & $\top$ & $\top$ \\
					\hline
					
				\end{tabular} 
				\qquad
				\begin{tabular}{|ccc|}
					\hline 
					$x$& $\bot$ &  $\top$\\
					\hline
					$\neg x$ & $\bot$ & $\bot$\\
					$\lrcorner x$ & $\top$ & $\bot$\\
					\hline
					
				\end{tabular}
				\caption{ Operations $\sqcap$, $\sqcup$, $\neg$ and $\lrcorner$ of $\underline{D}_{2,II}$}\label{OD2II}
			\end{figure}
			
			\item \textbf{\underline{Case 3}}: $(\top\sqcap\top,\bot\sqcup\bot)=(\bot,\top)$, then $\underline{D}_{2}$ is denoted by $\underline{D}_{2,III}$. 
			We have $\bot=\top$, therefore this algebra is not pure. We set $D_{2,III}:=\{\bot, a\}$. The Cayley's tables of $\underline{D}_{2,III}$  
			are  in Figure~\ref{OD2III}. Moreover, $D_{2,III_\sqcap}=D_{2,III_\sqcup}=\{\bot\}$.
			
			\begin{figure}[h]
				\begin{tabular}{|c|cc|}
					\hline 
					$\sqcap$ & $\bot$ & $a$\\
					\hline
					$\bot$ & $\bot$ & $\bot$  \\
					$a$ & $\bot$ & $\bot$ \\
					\hline
				\end{tabular}
				\qquad
				\begin{tabular}{|c|cc|}
					\hline 
					$\sqcup$ & $\bot$ & $a$\\
					\hline
					$\bot$ & $\bot$ & $\bot$  \\
					$a$ &$\bot$ & $\bot$ \\
					\hline
					
				\end{tabular} 
				\qquad
				\begin{tabular}{|ccc|}
					\hline 
					$x$& $\bot$ &  $a$\\
					\hline
					$\neg x$ & $\bot$ & $\bot$\\
					$\lrcorner x$ & $\bot$ & $\bot$\\
					\hline
					
				\end{tabular}
				\caption{ Operations $\sqcap$, $\sqcup$, $\neg$ and $\lrcorner$ of $\underline{D}_{2,III}$}\label{OD2III}
			\end{figure}
			\item \textbf{\underline{Case 4}}: $(\top\sqcap\top,\bot\sqcup\bot)=(\top,\bot)$, then  $\neg\bot=\lrcorner\bot$, $\neg\top=\lrcorner\top$, $\neg\neg\bot=\bot$ and $\neg\neg\top=\top$, therefore $\underline{D_{2}}$ is the \textit{two}-element  Boolean algebra $\underline{2}$.
		\end{itemize}
	\end{proof}

	\begin{remark}
		The double Boolean algebras $\underline{D}_{2,I}$, $\underline{D}_{2,II}$, $\underline{D}_{2,III}$ and $\textit{\underline{2}}$  are simple double Boolean algebras.
	\end{remark} 
	
	\begin{corollary}\label{co1}
		\begin{enumerate}
			\item  The 
			algebras $\underline{D}_{2,I}$ and $\underline{D}_{2,III}$ are up to isomorphism, the only  simple double Boolean algebras of type $I$ having more than one element.
			\item  The 
			algebras  $\underline{D}_{2,II}$ and $\underline{D}_{2,III}$ are up to isomorphism, the only  simple double Boolean algebras of type $II$ having more than one element.
			\item  The 
			algebras $\underline{D}_{2,I}$, $\underline{D}_{2,II}$ and $\underline{D}_{2,III}$  are up to isomorphism, the only  simple trivial  double Boolean algebras  having more than one element.
			\item The 
			algebras $\underline{D}_{2,I}$, $\underline{D}_{2,II}$ and $\underline{D}_{2,III}$  are up to isomorphism, the only  simple  double Boolean algebras of type $IV$ having more than one element.
			\item The 
			algebras $\textit{\underline{2}}$ and $\underline{D}_{2,III}$ are up to isomorphism, the only  simple  double Boolean algebras of type $V$ having more than one element. 
		\end{enumerate}
	\end{corollary}
	
	\begin{proof}
		Follows from Corollary \ref{c10} and Proposition \ref{p7}.
	\end{proof}
	
	The following lemma is the characterization of finite  sub-directly irreducible dBas given by Vormbrock. We can generalize this result to sub-directly irreducible algebras of type $l$, $l\in\{I, II, IV, V\}$. 
	\begin{lemma}\label{l2} \cite{11}
		A finite double Boolean algebra $\underline{D}$ is sub-directly irreducible if and only if $\underline{D}$ is simple.
	\end{lemma}
	
	
	\begin{theorem}\label{t20}
		
		\begin{enumerate}
			\item  The 
			algebras $\underline{D}_{2,I}$ and $\underline{D}_{2,III}$ are up to isomorphism, the only  sub-directly irreducible 
			dBas of type $I$ having more than one element.
			\item  The 
			algebras  $\underline{D}_{2,II}$ and $\underline{D}_{2,III}$ are up to isomorphism, the only  sub-directly irreducible 
			dBas of type $II$ having more than one element.
			\item The 
			algebras $\underline{D}_{2,I}$, $\underline{D}_{2,II}$ and $\underline{D}_{2,III}$  are up to isomorphism, the only  sub-directly irreducible  
			dBas of type $IV$ having more than one element.
			\item The 
			algebras $\textit{\underline{2}}$ and $\underline{D}_{2,III}$ are up to isomorphism, the only  sub-directly irreducible  
			dBas of type $V$ having more than one element. 
		\end{enumerate}
	\end{theorem}
	
	The proof of Theorem \ref{t20} is organized as follows: 
	items $(1)$ and $(2)$ 
	are in Proposition \ref{p20}, and the proof of item $(3)$ and $(4)$ in Proposition 
	\ref{t5} and \ref{t10} respectively.

	\begin{lemma}\label{l6}
		Let $\underline{D}$ be a dBa such that $|D_{p}|>1$. If $\underline{D}$ is sub-directly irreducible, then $\underline{D}_{p}$ is also sub-directly irreducible.
	\end{lemma}
	
	\begin{proof}
		We assume that $\underline{D}$ is sub-directly irreducible. 
		Since $D_{p}$ is a sub-algebra of $\underline{D}$, the restriction of every congruence relation of $\underline{D}$ on $D_{p}$ is also a congruence relation on $\underline{D}_{p}$. By Proposition \ref{t9}, we conclude that $\cap (Con(\underline{D})\setminus\{\Delta_{D}\})=(\cap (Con(\underline{D}_{p})\setminus\{\Delta_{D_{p}}\}))\cup\Delta_{D}$. $\underline{D}_{p}$ is sub-directly irreducible by Proposition \ref{t4}.
	\end{proof}
	Before continuing, we observe the following: 
	\begin{remark}\label{rem5}
		Let $\underline{D}$ be a dBa and $a$ and $b$ be two elements of $D$ such that $a\sqcap a=b\sqcap b$ and  $a\sqcup a=b\sqcup b$. Then the 
		relation $\theta=\Delta_{D}\cup\{(a,b), (b,a)\}$ is a congruence relation of $\underline{D}$.
	\end{remark}
	
	The next proposition characterizes all sub-directly irreducible dBas $\underline{D}$ of type $I$ and $II$.
	
	\begin{proposition}\label{p20} 
		
		\begin{enumerate}
			\item  The 
			algebras $\underline{D}_{2,I}$ and $\underline{D}_{2,III}$ are up to isomorphism, the only  sub-directly irreducible double Boolean algebra of type $I$ having more than one element.
			\item  The 
			algebras  $\underline{D}_{2,II}$ and $\underline{D}_{2,III}$ are up to isomorphism, the only  sub-directly irreducible double Boolean algebra of type $II$ having more than one element.
		\end{enumerate}
	\end{proposition}
	
	\begin{proof}
		\begin{enumerate}
			\item By Proposition \ref{p7}, $\underline{D}_{2,I}$ and $\underline{D}_{2,III}$ are up to isomorphism, the only  \textit{two}-element dBas of type $I$.  Moreover,  $\underline{D}_{2,I}$ and $\underline{D}_{2,III}$ are sub-directly irreducible. 
			
			Let $\underline{D}$ be a dBa of type $I$ such that $|D|\ge 3$. We distinguish two cases : (i) $\underline{D}$ is pure and (ii) $\underline{D}$ is not pure.
			
			\textbf{\underline{Case 1:}}  $\underline{D}$ is pure, then by Corollary \ref{c1}, $Con(\underline{D})\cong Con(\underline{D}_{\sqcap})$. Since $\underline{D}_{\sqcap}$ is not sub-directly irreducible ( due to Lemma \ref{l3}), we conclude that $\underline{D}$ is not sub-directly irreducible.
			
			\textbf{\underline{Case 2:}} $\underline{D}$ is not pure. We discuss on the cardinality of $\underline{D}_{p}$;
			
			\begin{itemize}
				\item If $|D_{p}|=1$, then  there are $a,b\in D\backslash D_{p} $ with $a\neq b$. Set $\theta_{1}=\{(\bot,a),(a,\bot)\}\cup \Delta_{D}~~\text{and}~~\theta_{2}=\{(\bot,b),(b,\bot)\}\cup \Delta_{D}$. From Remark \ref{rem5}, $\theta_{1},\theta_{2}\in Con(\underline{D})$. Moreover $\theta_{1}\neq\Delta_{D}$, $\theta_{2}\neq\Delta_{D}$ and $\theta_{1}\cap\theta_{2}=\Delta_{D}$. Thus by Proposition \ref{t4}, $\underline{D}$ is not sub-directly irreducible.
				
				\item If $|D_{p}|=2$ and $|D|=3$, then there is $a\in D\setminus D_{p}$. From Proposition \ref{t9}, $\theta=D_{p}^2\cup\{(a,a)\}$ is a congruence relation. Thus $\theta$ is in  $Con(\underline{D})\backslash\{\Delta_{D}, \nabla_{D} \}$. Hence, by Lemma \ref{l2}, $\underline{D}$ is not sub-directly irreducible.
				
				\item If $|D_{p}|=2$ and $|D|>3$, then $|D\setminus D_{p}|>1$. 
				
				First, we show that there exists $a,b\in D$ such that 
				$a\neq b, a\sqcap a=b\sqcap b$ and $a\sqcup a=b\sqcup b$. 
				If there exists $x\in D\setminus D_{p}$ such that $x\sqcap x=x\sqcup x$, then we choose $a=x\sqcap x$ and $b=x$. Else, for every $x\in D\setminus D_{p}$, $x\sqcap x\sqsubseteq x\sqcup x$ and $x\sqcap x\neq x\sqcup x$. Therefore, for every $x\in D\setminus D_{p}$, $x\sqcap x=\bot$ and $x\sqcup x=\top$ ( due to $D_{p}=\{\bot,\top\}$). Thus $a$ and $b$ are just two different elements of $D\setminus D_{p}$. 
				
				Second, we show that $D$ is not sub-directly irreducible. Let $a,b\in D\setminus D_{p}$ such that $a\sqcap a=b\sqcap b$, $a\sqcup a=b\sqcup b$, and $a\neq b$. Set $\theta=\Delta_{D}\cup\{ (a,b),(b,a)\}$ and $\beta=D_{p}^{2}\cup\Delta_{D}$. From Remark \ref{rem5}, $\theta$ is a congruence relation. Moreover $\theta\neq\Delta_{D}$, $\beta\neq\Delta_{D}$ and $\theta\cap\beta=\Delta_{D}$. Thus, by Proposition \ref{t4}, $\underline{D}$ is not sub-directly irreducible. 
				
				\item If $|D_{p}|\ge 3$,  then applying the result obtained in case 1, $\underline{D}_{p}$ is not sub-directly irreducible. The contra-position of Lemma \ref{l6} yields that $\underline{D}$ is not sub-directly irreducible.
			\end{itemize}
			\item Dual of (1).
			
		\end{enumerate}
	\end{proof}
	
	The following proposition characterizes trivial sub-directly irreducible dbas.
	
	\begin{proposition}\label{p21}
		The double Boolean algebras $\underline{D}_{2,I}$, $\underline{D}_{2,II}$ and $\underline{D}_{2,III}$  are up to isomorphism, the only  sub-directly irreducible trivial  double Boolean algebras  having more than one element.
	\end{proposition}
	
	\begin{proof}
		By Proposition \ref{p7}, $\underline{D}_{2,I}$, $\underline{D}_{2,II}$ and $\underline{D}_{2,III}$ are up to isomorphism, the only trivial \textit{two}-element dBas . Moreover,  $\underline{D}_{2,I}$, $\underline{D}_{2,II}$ and $\underline{D}_{2,III}$ are sub-directly irreducible.
		
		Let $\underline{D}$ be a trivial dBa such that $|D|\ge 3$.
		If $\underline{D}$ is of type $I$ or type $II$, then $\underline{D}$ is not sub-directly irreducible by Proposition \ref{p20}.
		Assume that $\underline{D}$ is neither of type $I$ nor of type $II$. We distinguish again two cases as in the proof of the previous proposition. :
		
		\textbf{\underline{Case 1:}} $\underline{D}$ is pure.
		Set $J=\{\bot\}$, $G=D_{\sqcap}$, $F=D_{\sqcup}$ and $H=\{\top \}$. $(J,F)$ and $(G,H)$ are two different nontrivial congruence generating pairs and $(J,F)\wedge (G,H)=(\{\bot\}, \{\top\})$.  Hence, by Proposition \ref{t4}, $\underline{D}$ is not sub-directly irreducible.
		
		\textbf{\underline{Case 2:}} $\underline{D}$ is not pure. Then, $|D_{p}|\ge 3$. Applying the result obtained in case 1, we conclude that $\underline{D}_{p}$ is not sub-directly irreducible. Hence, by Lemma \ref{l6}, $\underline{D}$ is not sub-directly irreducible.
		
	\end{proof}
	
	We continue with the characterization of sub-directly irreducible dBas $\underline{D}$ of type $IV$.
	\begin{proposition} \label{t5}
		The 
		algebras $\underline{D}_{2,I}$, $\underline{D}_{2,II}$ and $\underline{D}_{2,III}$  are up to isomorphism, the only  sub-directly irreducible  
		dBas of type $IV$ having more than one element.
	\end{proposition}
	
	\begin{proof}
		From Proposition \ref{p7}, $\underline{D}_{2,I}$, $\underline{D}_{2,II}$ and $\underline{D}_{2,III}$ are up to isomorphism, the only   \textit{two}-element dBas of type $IV$. Moreover,  $\underline{D}_{2,I}$, $\underline{D}_{2,II}$ and $\underline{D}_{2,III}$ are sub-directly irreducible.
		
		Let $\underline{D}$ be a dBa of type $IV$ such that $|D|\ge 3$.
		If $\underline{D}$ is trivial, then $\underline{D}$ is not sub-directly irreducible by Proposition \ref{p21}.
		Assume that $\underline{D}$ is not trivial. We distinguish again two cases :
		
		\textbf{\underline{Case 1:}} $\underline{D}$ is pure.  
		\begin{itemize}
			\item If $\bot\sqcup\bot\neq \bot $, then, set $J=I(\bot\sqcup\bot)$, $F=\{\top \}$, $G=I(\neg(\bot\sqcup\bot))$ and $H=D_{\sqcup}$.
			It is easy to see that $(J,F)$ and $(G,H)$ are two nontrivial congruence generating pairs such that the meet is $(\{\bot\}, \{\top\})$. Hence, again by Proposition \ref{t4}, $\underline{D}$ is not sub-directly irreducible.
			
			\item If $\top\sqcap\top\neq\top$, then, similarly as in the previous sub-case, we obtain that $\underline{D}$ is not sub-directly irreducible.
		\end{itemize}
		
		\textbf{\underline{Case 2:}} $\underline{D}$ is not pure. Then, $|D_{p}|\ge 3$. We use the result of case 1, and Lemma \ref{l6} to conclude.	
	\end{proof}

	We end this part by the characterization of sub-directly irreducible dBas of type $V$.
	
	\begin{proposition}\label{t10}
		The double Boolean algebras  $\mathit{\underline{2}}$ and $\underline{D}_{2,III}$ are up to isomorphism, the only  sub-directly irreducible  double Boolean algebras of type $V$ having more than one element. 
	\end{proposition}
	
	\begin{proof}
		
		It is similar to the proof of Proposition \ref{p20}.	
	\end{proof}
	We have then finished proving Theorem \ref{t20}.
	Our results in this section are summarized in the following table.
	\begin{table}[h]
		\begin{tabular}{|l|l|l|l|l|l|}
			\hline 
			& dBas of  & dBas of& Trivial&dBas of &dBas of\\
			&  type $I$ &  type $II$& dBas& type $IV$& type $V$\\
			\hline
			Simple & $\underline{D}_{2,I}$ and   & $\underline{D}_{2,II}$ and &$\underline{D}_{2,I}$, $\underline{D}_{2,II}$ &$\underline{D}_{2,I}$, $\underline{D}_{2,II}$ &$\underline{\textit{2}}$ and    \\
			& $\underline{D}_{2,III}$  &  $\underline{D}_{2,III}$& and $\underline{D}_{2,III}$& and $\underline{D}_{2,III}$& $\underline{D}_{2,III}$   \\
			\hline
			Sub-directly &  $\underline{D}_{2,I}$ and   & $\underline{D}_{2,II}$ and &$\underline{D}_{2,I}$, $\underline{D}_{2,II}$ &$\underline{D}_{2,I}$, $\underline{D}_{2,II}$&$\underline{\textit{2}}$ and   \\
			irreducible&   $\underline{D}_{2,III}$  &  $\underline{D}_{2,III}$& and $\underline{D}_{2,III}$& and $\underline{D}_{2,III}$& $\underline{D}_{2,III}$  \\
			\hline
		\end{tabular}
		
		\caption{Simple and sub-directly irreducible dBas}
	\end{table}

	\section{Conclusion and Further Research}
	In this work, 
	we have characterized pure and trivial double Boolean algebras as glued sum of two Boolean algebras. We have also characterized simple double Boolean algebras, and determined up to isomorphism all sub-directly irreducible double Boolean algebras of the class of trivial double Boolean algebras, the class of  type $I$, $II$, $IV$ and $V$ as specified in the document. The characterization of infinite sub-directly irreducible double Boolean algebras of type $III$ is still open and  
	will be addressed 
	in 
	future work.
	
	\subsection*{Data availability}
	Data sharing not applicable to this article as datasets were neither generated nor analyzed.
	
	\subsection*{Compliance with ethical standards}
	The authors declare that they have no conflict of interest.


\end{document}